\documentclass[11pt,reqno]{amsart}
\usepackage{amssymb}
\usepackage{latexsym}
\newtheorem{theorem}{Theorem}
\newtheorem{lemma}[theorem]{Lemma}

\newtheorem{conj}[theorem]{Conjecture}

\newenvironment{proof*}{\noindent {\it Proof.~~}\ }{}

\newcommand{\beeq}{\begin{eqnarray*}}
\newcommand{\eneq}{\end{eqnarray*}}

\def\Z{\mathbb Z}
\def\zp{{\mathbb Z}/p{\mathbb Z}}

\newcommand{\ot}{\overline{T}}

\newcommand{\be}{\begin{equation}}
\newcommand{\ee}{\end{equation}}

\title[Large sets with small doubling modulo $p$]{Large
sets with small doubling modulo $p$ are well covered by an arithmetic progression}

\author{Oriol Serra}
\address{Universitat Polit\`ecnica de Catalunya\\
Matem\`atica Aplicada IV\\
Campus Nord - Edif. C3, C. Jordi Girona, 1-3\\
08034 Barcelona, Spain.}
\email{oserra@ma4.upc.edu}

\author{Gilles Z\'emor}
\address{Institut de Math\'ematiques de Bordeaux\\
Universit\'e de Bordeaux 1\\
351 cours de la Lib\'eration, 33405 Talence \\
France.}
\email{zemor@math.u-bordeaux1.fr}

\begin{document}
\baselineskip 16pt

\begin{abstract} We prove that there is  $\epsilon>0$ and $p_0>0$
such that for every   prime $p>p_0$, every subset $S$ of $\Z/p\Z$
which satisfies $|2S|\le (2+\epsilon)|S|$ and $2(|2S|)-2|S|+3\le p$ is
contained in an arithmetic progression of length $|2S|-|S|+1$. This
is the first result of this nature which places no unnecessary
restrictions on the size of $S$.
\end{abstract}

%MSC Class: 11P70

\maketitle

\section{Introduction}

In 1959 Freiman \cite{fre59} proved that if $S$ is a set of integers such
that
$$|2S|\le 3|S|-4$$
then $S$ is contained in an arithmetic progression of length
$|2S|-|S| +1$.

This result is often known as Freiman's $(3k-4)$ Theorem. It
has been conjectured that the same result also holds in the finite
groups $\Z/p\Z$ of prime order.  Working towards this conjecture,
Freiman \cite{fre59} proved~(see Nathanson \cite{Nat} for the
following formulation of the result):

\begin{theorem}[Freiman \cite{fre59}]\label{thm:fv} Let $S\subset
\Z/p\Z$ such that $3\le |S|\le c_0p$ and
$$
|2S|\le c_1 |S|-3,
$$
with $0<c_0\le 1/12$, $c_1>2$ and $(2c_1-3)/3<(1-c_0c_1)/c_1^{1/2}$.
Then $S$ is contained in an arithmetic progression of length
$|2S|-|S|+1$.
\end{theorem}

The largest possible numerical value of $c_1$ given by this theorem is
$c_1\approx 2.45$, which falls somewhat short of the value predicted
by the conjecture (namely $3$). In addition, 
Theorem~\ref{thm:fv} 
only guarantees the result for sets $S$ that are small enough.
For example, to guarantee $c_1=2.4$, the theorem needs the assumption
$|S|\leq p/35$. This last assumption was improved to $|S|\leq p/10.7$
by R{\o}dseth \cite{ROD} but without improving the value of the constant $c_1$.

It follows from a recent result of Green and Rusza \cite{gr} on
rectification of sets with small doubling in $\Z/p\Z$ that the value
of $c_1$ can actually be pushed all the way to $3$ while preserving the
conclusion that
$S$ is contained in a short arithmetic progression, but this comes at the
expense of a stringent condition on the size of $S$: namely the
the extra assumption $|S|<10^{-180}p.$

In the present paper, we shall work at the conjecture from a
different direction. Rather than focusing on the best possible value
for the constant $c_1$, we shall try to lift all restrictions on the
size of $S$. First we need to formulate properly what should be the
right version of Freiman's $(3k-4)$ theorem in $\Z/p\Z$.

For $-1\leq m\leq |S|-4$, we want the condition $|2S|= 2|S|+m$ to
imply that $S$ is included in an arithmetic progression of length
$|S|+m+1$. One fact that has not been spelt out explicitly in the
literature is that for such a result to hold, some lower bound on
the size of the {\em complement} $\Z/p\Z\setminus 2S$ of $2S$ must
be formulated. Indeed, if $p-|2S|$ is too small, the conclusion will
not hold even if $m$ is small compared to $|S|-4$. Consider in
particular the following example. Let $S=\{ 0\} \cup \{m+3, m+4,
\ldots ,(p+1)/2\}$. We have $|2S|=p-(m+1)= 2|S|+m$, but it can be
seen with a little thought  that   $S$   is not included in an
arithmetic progression  of length $|S|+m+1$. For the desired result
to hold, we must therefore add the condition $p-|2S|>m+1$. We conjecture
that this extra condition is sufficient for a $\Z/p\Z$-version of
Freiman's $(3k-4)$ theorem to hold. More precisely~:

\begin{conj}\label{conj}
  Let $S\subset\Z/p\Z$ and let  $m=|2S|-2|S|$. Suppose that $m$
satisfies~:
$$
-1\leq m \leq \min\{ |S|-4, p-|2S|-3\}.
$$
Then $S$ is included in an arithmetic progression of length
$|S|+m+1$.
\end{conj}
Note that $p-|2S|=p-2|S|-m$ can not be equal to $m+2$, otherwise
$p$ would be an even number. Therefore condition (ii) of the
conjecture is equivalent to $p-|2S|>m+1$, as implied by the
example above.

We remark that the cases $m=-1,0,1$ of this conjecture are known.
They are implied by Vosper's theorem \cite{vosper} ($m=-1$), by a
result of  Hamidoune and R{\o}dseth \cite{HR} ($m=0$) and by a
result of   Hamidoune and the present authors   \cite{GOY} ($m=1$).
 In the present paper we shall prove
conjecture~\ref{conj} for all values of $m$ up to $\epsilon |S|$,
where $\epsilon$ is a fixed absolute constant. More precisely, our
main result is~:

\begin{theorem}\label{thm:main}
There exist positive numbers $p_0$ and $\epsilon$ such that, for all
primes $p>p_0$, any subset $S$ of $\Z/p\Z$ such that
\begin{itemize}
\item[(i)]
$|2S| < (2 + \epsilon)|S|$,
\item[(ii)]
$m=|2S|-2|S|$ satisfies $m\leq \min \{ |S|-4, p-|2S|-3\}$,
\end{itemize}
is included in an arithmetic progression of length $|S|+m+1$.
\end{theorem}
We shall prove this result with the numerical values
$\epsilon=10^{-4}$ and $p_0=2^{94}$.

In the past, the dominant strategy, already present in Freiman's original proof
of Theorem \ref{thm:fv},  has been to {\em rectify} the set $S$,
i.e. find an argument that enables one to claim that the sum $S+S$
must behave as in $\Z$, and then apply Freiman's $(3k-4)$
theorem.  Rectifying $S$ directly however, becomes more and
more difficult when the size of $S$ grows, hence the different upper
bounds on $S$ that one regularly encounters in the literature. In
our case, without any upper bound on $S$, rectifying $S$ by studying
its structure directly is a difficult challenge. Our method will be
indirect. Our strategy is to use an auxiliary set $A$ that minimizes
the difference $|S+A|-|S|$ among all sets such that $|A|\geq m+3$.
The set $A$ is called an $(m+3)$-atom of $S$ and using such sets to
derive properties of $S$ is an instance of the isoperimetric (or
atomic) method in additive number theory which was introduced by
Hamidoune and developed in \cite{HALG,HWAR,HACTA,GOY,GOY2}. The
point of introducing the set $A$ is that we shall manage to prove
that it is both significantly smaller than $S$ and also has a small
sumset $2A$. This will enable us to show that first the sum $A+A$, and
then the sum $S+A$, must behave as in $\Z$. Finally we will use
 Lev and Smelianski's distinct set version \cite{ls} of
Freiman's $(3k-4)$ Theorem to conclude.

The paper is organised as follows. The next section will introduce
$k$-atoms and their properties that are relevant to our purposes. In
Section 3 we will show how our method works proving Theorem
\ref{thm:main} in the relatively easy case when $m$ is an arbitrary
constant or a slowly growing function of $p$ (i.e. $\log p$). In
Section 4 we will prove Theorem \ref{thm:main} in full when $m$ is a
linear function of $|S|$.

\section{Atoms}
Let  $S$ be a
subset of $\Z/p\Z$ such that $0\in S$. For a positive integer $k$, we
shall say that $S$ is $k$-{\em separable} if there exists
$X\subset \Z/p\Z$ such that $|X|\geq k$ and  $|X+S|\leq p-k$.

Suppose that $S$ is $k$-{separable}. The
$k$-th {\em isoperimetric number} of $S$  is then defined by
\begin{equation}  \label{eq:kappa}
\kappa _k (S)=\min  \{|X+S|-|X|,\; | \ \
X\subset\Z/p\Z, \ |X|\geq k \ {\rm and }\ |X+S|\leq p-k\}.
\end{equation}

 For a $k$-separable
set $S$, a subset $X$ achieving the above minimum is called a
$k$-{\em fragment} of $S$. A $k$-fragment with minimal cardinality
is called a $k$-{\em atom}.

What makes $k$-atoms interesting objects is the following lemma~:

\begin{lemma}[The intersection property \cite{HALG}]
  \label{lem:intersection}
Let  $S$ be a subset of $\Z/p\Z$ such that $0\in S$, and
suppose $S$ is $k$-separable.
Let $A$  be a  $k$-atom of $S$. Let  $F$  be a $k$-fragment of $S$
such that  $ A \not\subset F $.  Then   $|A \cap F|\leq k-1 .$
\end{lemma}

The following Lemma is proved in \cite{GOY2}:

\begin{lemma}\label{lem:m+3} Let $S\subset \Z/p\Z$ with $|S|\ge
  3$ and $0\in S$. Suppose $S$ is $2$-separable and
$\kappa_2 (S)\le |S|+m$. Let $A$ be a $2$--atom of $S$. Then $|A|\le
m+3$.
\end{lemma}

Lemma~\ref{lem:m+3} implies the following upper bound on the size of
atoms.

\begin{lemma}\label{lem:3m+5}
Let $k\geq 3$ and let $A$ be a $k$--atom of a $k$--separable set
$S\subset \Z/p\Z$ with $0\in S$, $|S|\ge 2$ and $\kappa_k (S)\le
|S|+m$. Then $|A|\le 2m+k+2$.
\end{lemma}

\begin{proof} 
The set $A$ is clearly $2$--separable.
Let $B$ be a $2$--atom of $A$ with $0\in B$, so that
$|B+A|\le |B|+|A|+m$. Let $b\in B$, $b\neq 0$. By Lemma
\ref{lem:m+3} we have $|B|\le m+3$. Therefore, \be \label{eq:3m+5}
|A\cup (b+A)|=|\{ 0,b\}+A|\le |B+A|\le |A|+2m+3. \ee But $b+A$ is
also a $k$--atom of $S$. By the intersection property, it follows
that $|A\cap (b+A)|\le k-1$. Hence $2|A|-(k-1)\le |A\cup (b+A)|$
which together with (\ref{eq:3m+5}) gives the result.
\end{proof}

From now on $S$ will refer to a subset of $\Z/p\Z$ satisfying
conditions (i) and (ii) of Theorem~\ref{thm:main} for a fixed
$\epsilon >0$ to be determined later, and $m$ always denotes the
integer $m=|2S|-|S|$. Without loss of generality we will also assume
$0\in S$.

Note that condition  (ii) implies that $S$ is $(m+3)$--separable so
that $(m+3)$-atoms of $S$ exist. Note that by the definition of an
atom, if $X$ is an atom of $S$ then so is $x+X$ for any
$x\in\Z/p\Z$. Therefore there are atoms containing the zero element.

In the sequel
$A$ will denote an $(m+3)$--atom of $S$ with $0\in A$.
We will regularly call upon the following two inequalities:
\be
\label{eq:ubsa} |S+A|\le |S|+|A|+m
\ee
which follows from the definition of an atom, and
\be\label{eq:|A|}
  |A|\le 3m+5.
\ee
which follows from Lemma \ref{lem:3m+5} with $k=m+3$.

The reader should also bear in mind that for all practical purposes,
inequality \eqref{eq:|A|} means that we will only be dealing with
cases when $|A|$ is significantly smaller than $|S|$. Indeed, we
shall prove Theorem~\ref{thm:main} for a small value of $\epsilon$,
namely $\epsilon=10^{-4}$, so that $3m$ is very much smaller than
$|S|$. We can also freely assume that $|S|\geq p/35$, since
otherwise Freiman's Theorem~\ref{thm:fv}  gives the result with
$\epsilon =0.4$. The prime $p$ will also be assumed to be larger
than some fixed value $p_0$ to be determined later.

\section{The case $m\leq \log p$}
In this section we will deal with the case when $m$ is a very small
quantity, i.e. smaller than a logarithmic function of $p$.
This will allow us to introduce, without technical difficulties
to hinder us, the general idea of the method which is to first
show that $A$ must be contained in a short arithmetic progression
and then to transfer the structure of $A$ to the larger set $S$.
It will also serve the additional purpose of allowing us to suppose
$m\geq 6$ when we switch to the looser condition $m\leq\epsilon |S|$.

We start by stating some results that we shall call upon. The first is
a generalization of Freiman's theorem in $\Z$ to sums of
different sets and is proved by Lev and
Smelianski in \cite{ls}, we give it here somewhat reworded
(see also \cite[Th. 5.12]{tv}).

\begin{theorem}[Lev and Smelianski \cite{ls}]\label{thm:ls2}
Let $X$ and $Y$ be two
nonempty finite sets of integers with
$$
|X+Y|= |X|+|Y|+\mu.
$$
Assume that    $\mu\le \min \{|X|, |Y|\}-3$ and that one of the two
sets $X,Y$ has size at least $\mu +4$. Then  $X$ is contained in an
arithmetic progression of length $|X|+\mu +1$ and $Y$ is contained
in an arithmetic progression of length $|Y|+\mu +1$.
\end{theorem}

The second result we shall use is due to Bilu, Lev and Ruzsa
\cite[Theorem 3.1]{blr}\footnote{In \cite{blr} their statement is
slightly different from Theorem \ref{thm:rectlev}, 
but this is actually what they prove.} 
 and gives    a bound on the length of small
sets in $\Z/p\Z$. By the {\em length}  $\ell (X)$ of a set $X\subset
\Z/p\Z$ we mean the length (cardinality) of the shortest arithmetic progression
which contains $X$.

\begin{theorem}[Bilu, Lev, Ruzsa \cite{blr}]\label{thm:rectlev}
Let $X\subset \zp$ with $|X|\le \log_4 p$. Then $\ell (X) < p/2$.
\end{theorem}

Theorem \ref{thm:rectlev} will be used to show that, when $m$ is
small enough, then the atom $A$ is contained in a short arithmetic
progression.

\begin{lemma}\label{lem:rectasmall} Suppose that
$6m+11\le   \log_4 p$. Then   $A$ is contained in an arithmetic
progression of length $2(|A|-1)$.
\end{lemma}

\begin{proof}
Since we assume $|S|\ge p/35$, it follows from  \eqref{eq:ubsa} and
\eqref{eq:|A|}   that  $A$ is an $(m+4)$--separable set. Let
therefore $B$ be an $(m+4)$--atom of $A$ containing $0$, so that
$|B+A|\le |B|+|A|+m$. By Lemma~\ref{lem:3m+5} we have $|B|\le 3m+6$
so that $|A\cup B|\le 6m+11$. By the present lemma's hypothesis, it
follows from Theorem~\ref{thm:rectlev} that $A\cup B$ is contained
in an arithmetic progression of length less than $p/2$. 
The sum $A+B$ can therefore be considered as a sum of integers, so that
Theorem~\ref{thm:ls2} applies and $A$ is contained in an arithmetic
progression of length $|A|+m+1\le 2|A|-2$.
\end{proof}

We now proceed to deduce from Lemma~\ref{lem:rectasmall} the structure
of $S$. It will be convenient to introduce the following notation.

Recall that we denote by $\ell (X)$ the length of the smallest
arithmetic progression containing $X$. By $\ell_X(Y)$ we shall
denote the length of a smallest arithmetic progression of difference
$x$ containing $Y$, where $x$ is the difference of a shortest
arithmetic progression containing $X$.

The point of the above definition is that if we have
$\ell_A (S)+\ell (A)\leq p$ then the sum $S+A$ can be considered
as a sum in $\Z$, so that (\ref{eq:ubsa}) and Theorem
\ref{thm:ls2} applied to $S$ and $A$ imply Theorem
\ref{thm:main}. We summarize this point in the next Lemma for future reference.

\begin{lemma}
  \label{lem:l(S)+l(A)}
If we can assume $\ell_A (S)+\ell (A)\leq p$ then Theorem \ref{thm:main} holds.
\end{lemma}

Whenever we will wish transfer the structure of $A$ to $S$ we will
assume that $\ell_A (S)+\ell(A)> p$ and look for a contradiction.
We can think of this hypothesis as $S$ having no `holes' of
length $\ell(A)$. In the present case of very small $m$, the desired
result on $S$ follows with very little effort.

\begin{lemma}\label{lem:rectsasmall}
Suppose that $ 6m+11\le  \log_4 p $. Then $S$ is contained in an
arithmetic progression of length $|S|+m+1$.
\end{lemma}

\begin{proof} By Lemma \ref{lem:rectasmall},
$A$ is contained in an arithmetic progression
of difference $r$, that
we can assume to equal $r=1$, and of length $2(|A|-1)$.
In particular $A$
has two consecutive elements. Without loss of generality we may
replace $A$ by a translate of $A$ and
assume that $\{ 0,1\}\subset A$.  Let $S=S_1\cup\cdots \cup S_k$ be
the decomposition of $S$ into maximal arithmetic progressions of
difference one, so that
$$
|S+A|\ge |S|+k.
$$
Because of (\ref{eq:ubsa}) we have $k\le |A|+m$.
By Lemma~\ref{lem:l(S)+l(A)} we can assume
every maximal
arithmetic progression in the complement of $S$ to have length at most
$\ell (A)$. Therefore, 
$$
\ell_A (S) +\ell(A)\le |S|+k\ell (A)\le |S|+(|A|+m)2(|A|-1).
$$
Now by \eqref{eq:|A|} we get
$$\ell_A (S) +\ell(A)\le |S|+(4m+5)(6m+8)<
|S|+(\log_4p )^2<\frac p2 + ( \log_4p)^2$$ since $|S|<p/2$. 
We have $\log_4^2p< p/2$ for all $p$ therefore 
we get $\ell_A (S) +\ell(A)< p$, a contradiction.
\end{proof}

\section{The general case}
\subsection{Overview}
When $m$ grows we encounter two difficulties. First,
Theorem~\ref{thm:rectlev} will not apply anymore to any set
containing $A$, and we need an alternative method to argue that $A$
is contained in a short arithmetic progression.
Second, even if we do manage to prove that $A$ is contained in a short
arithmetic progression, we will not be able to deduce the structure of $S$ from
\eqref{eq:ubsa} by the simple technique of the preceding section.

We will now use an extra tool, namely the Pl\"unecke-Ruzsa estimates
for sumsets; see e.g. \cite{ruzsa, Nat}.

\begin{theorem}[Pl\"unecke-Ruzsa \cite{ruzsa}]\label{thm:pr} Let $S$ and $T$ be
finite subsets of an abelian group with $|S+T|\le c|S|$. There is a
nonempty subset $S'\subset S$ such that
$$
|S'+jT|\le c^{j}|S'|.
$$
\end{theorem}

The Pl\"unecke-Ruzsa inequalities applied to $S$ and $A$ will give
us that there exists a positive $\delta$ such that either $A$ is
contained in a progression of length $(2-\delta)(|A|-1)$ or $2A$ is
contained in an arithmetic progression of length
$(2-\delta)(|2A|-1)$ (Lemma~\ref{lem:rectabis}). We will then
proceed to transfer the structure of $A$ or $2A$ to $S$.

Again we shall use Lemma~\ref{lem:l(S)+l(A)} to assume  that  $S$
does not contain a  ``gap'' of length $\ell (A)$ or $\ell (2A)$. We
define the density of a set $X\subset\Z/p\Z$ as $\rho
(X)=(|X|-1)/\ell (X)$. If $\ell (A)\le (2-\delta)(|A|-1)$ we will
argue that the sum $S+A$ must have a {\em density} at least that of
$A$ and get a contradiction with the upper bound on $|S+A|$. The
details will be given in Subsection \ref{subsec:l(A)}.

We will not be quite done however, because we can not guarantee that
$\ell(A)\leq (2-\delta)(|A|-1)$ holds. In that case we have to fall
back on the condition $\ell(2A)\leq (2-\delta)(|2A|-1)$, meaning
that it is the set $2A$, rather than $A$, that has large enough
density. In this case we have to work a little harder. We proceed in
two steps: we first apply the Pl\"unecke-Ruzsa inequalities again to
show that there exists a {\em large} subset $T$ of $S$ such that
$|T+2A|$ is small. We then apply the density argument to show that
$T$ must be contained in an arithmetic progression with few missing
elements.  We then focus on the remaining elements of $S$, i.e. the
set $S\setminus T$. We will again argue that if this set has a gap
of length $\ell(A)$ the desired result holds and otherwise the
density argument will give us that $S+A$ is too large. This analysis
is detailed in Subsection \ref{subsec:l(2A)} and will conclude our
proof of Theorem~\ref{thm:main}.

\subsection{Structure of $A$}\label{subsec:rect}
\begin{lemma}\label{lem:upperbound}
Suppose $6\leq m \leq \epsilon |S|$ with $\epsilon\leq 10^{-4}$.
Then for any positive integer $k\leq 32$ we have
$$
|kA|\le k (|A|+m)\left ( 1+\frac{5k\epsilon}{2}\right) + 1.
$$
\end{lemma}

\begin{proof}  Rewrite \eqref{eq:ubsa} as
$$
|S+A|\le |S|+|A|+m=c|S|,
$$
with $c=1+\frac{|A|+m}{|S|}$. By Theorem~\ref{thm:pr}
(Pl\"unecke--Ruzsa), for each $k$ there is a subset $S'=S'(k)$ such that
\be\label{eq:plu1}
|S'+kA|\le c^{k}|S'|.
\ee
Apply \eqref{eq:|A|} and $m\geq 6$ to get
$|A|\le 3m+5\le 4m$. Since  $m\le \epsilon |S|$ we obtain
for the constant $c$ just defined $c\leq 1+5\epsilon$.
We clearly have
$$c^{k}|S'|\le c^k|S|\le (1+5\epsilon)^k|S|<2|S|<p$$ for $k\le 32$.
Now apply the Cauchy-Davenport Theorem to $S'+kA$ in (\ref{eq:plu1})
to obtain $|S'|+|kA|-1\leq c^k|S'|$, from which
\begin{equation}
  \label{eq:kA}
  |kA|\le (c^k-1)|S'| + 1\le (c^k-1)|S| + 1.
\end{equation}
Numerical computations give that
   $$(1+x)^k\leq 1+kx+\frac{k^2}{2}x^2$$
for any positive real number $x\leq 5.10^{-4}$ and for $k\leq 32$.
Hence, since $c= 1+(|A|+m)/|S|\le 1+5\epsilon$,  we can
write, for $k\le 32$,
$$
c^{k} =\left(1+\frac{|A|+m}{|S|}\right)^k \le
1+k\frac{|A|+m}{|S|}+\frac{k^2}{2}\left(\frac{|A|+m}{|S|}\right)^2.
$$
Applied to \eqref{eq:kA} we get
\begin{eqnarray*}
 |kA|&\le&
 k(|A|+m)+\frac{k^2}{2}\left(\frac{(|A|+m)^2}{|S|}\right)+1 \\
 &\le &k(|A|+m)\left( 1+\frac{k}{2}\frac{(|A|+m)}{|S|}\right)+1 \\
 &\le&k(|A|+m)\left ( 1+\frac{5k\epsilon}{2}\right)+1,
 \end{eqnarray*}
as claimed.
\end{proof}

\begin{lemma}\label{lem:recta} If $6\leq m\leq \epsilon |S|$
with $\epsilon\leq 10^{-4}$, then $A$ and $2A$ are contained in an
arithmetic progression of length less than $p/2$.
\end{lemma}

\begin{proof} Put $k=2^j$ and $c_1=2.44$. Suppose that  $|2^jA|\ge
c_1|2^{j-1}A|-3$
 for each $1\le j\le 5$. Then,
 $$
 |32 A|\ge c_1^5|A|-3(c_1^5-1)/(c_1-1)\ge 86|A|-179\ge 65|A|+10,
 $$
 where in the last inequality we have used $|A|\ge m+3\ge 9$.
On the other hand, by  Lemma \ref{lem:upperbound}, we have
\be\label{eq:small2} |kA|\le k(|A|+m)\left(
1+\frac{5k\epsilon}{2}\right)+1\le 2k(1+\frac{5k\epsilon}{2})|A|,\ee
 which, for $k=32$, gives $|32A|\le 64(1+80\epsilon)|A|\le 65|A|$, a contradiction.

Hence $|2^jA|\le c_1|2^{j-1}A|-3$ for some $1\le j\le 5$. Since
$$|2^{j-1}A|\le |16A|\le 32(1+40\epsilon)|A|\le
64(1+40\epsilon)\epsilon p<8\cdot 10^{-3}p,$$ where again we have
used inequality (\ref{eq:small2}) for $k=16$ and $|A|\le 4m\le
4\epsilon |S|\le 2\epsilon p$. It follows from Freiman's Theorem
\ref{thm:fv} (with $c_0=8\cdot 10^{-3}$ and $c_1=2.44$) that
$A\subset 2^{j-1}A$ is contained in an arithmetic progression of
length at most
$$|2^jA|-|2^{j-1}A|+1< 1.44|2^{j-1}A|\le (1.44)8\cdot 10^{-3}p.$$
In particular, $A$ and $2A$ are included in arithmetic progressions
of lengths less than $p/2$.
\end{proof}

Now that we know that $A$ and $2A$ are contained in an arithmetic
progression of length smaller than $p/2$, we can apply to them the
Freiman's $(3k-4)$ Theorem to get the following result.

\begin{lemma}\label{lem:rectabis}
Suppose $6\leq m\leq \epsilon |S|$ with $\epsilon\leq 10^{-4}$,
and let
$0<\delta \leq 10^{-1}$. If $A$ is {\em not} contained in an
arithmetic progression of length  $(2-\delta)(|A|-1)$ then
$2A$ is contained in
 an arithmetic progression of length $(2-\delta)(|2A|-1)$.
\end{lemma}

\begin{proof}
Suppose first
that $|2A|\ge (3-\delta)(|A|-1)$ and $|4A|\ge (3-\delta)(|2A|-1)$.
Then
\begin{equation}
  \label{eq:4A}
  |4A|\ge (3-\delta)^2|A|-(3-\delta)^2-(3-\delta)\ge
(3-\delta)^2|A|-12.
\end{equation}
On the other hand, Lemma \ref{lem:upperbound} for $k=4$ and
$\epsilon=10^{-4}$ gives $|4A|\le 4(1+10\epsilon)(|A|+m)+1$. By
using \eqref{eq:4A} and   $m\leq |A|-3$   we get
$$(3-\delta)^2|A|-12 \leq 8(1+10\epsilon)|A| -12(1+10\epsilon) +1.$$
Since   $m\geq 6$, we have $|A|\geq m+3\geq 9$. Therefore we obtain
$$(3-\delta)^2|A|< \left(8(1+10\epsilon) + \frac 19\right)|A|,$$
a contradiction for $\delta \leq 0.1$.

Hence,
\begin{itemize}
\item[(a)] either $|2A|< (3-\delta)(|A|-1) < 3|A|-3$, but since
 $\ell (A)<p/2$   by Lemma~\ref{lem:recta}, 
Freiman's $(3k-4)$ Theorem
applies and $A$ is contained in an arithmetic progression of length
$|2A|-(|A|-1)\leq (2-\delta)(|A|-1)$.
\item[(b)] Or $|4A|< (3-\delta)(|2A|-1)< 3|2A|-3$, but using
Lemma~\ref{lem:recta} again,   
$(3k-4)$--Freiman's Theorem
implies that $2A$ is contained in an arithmetic progression of
length $(2-\delta)(|2A|-1)$.
\end{itemize}
\end{proof}

\subsection{Structure of $S$ when $\ell (A)$ is small.}\label{subsec:l(A)}

For a subset $B\subset \Z/p\Z$  define the {\em density} of $B$ by
$$\rho B = \frac{|B|-1}{\ell (B)}.$$ The next lemma gives a lower bound for the
cardinality of a sumset of two subsets $B, C\in \zp$ when $\ell
(B)+\ell (C) > p$ in terms of their densities.

\begin{lemma}\label{lem:dens}
Let $0\in C\subset \zp$ with  $C\subset [0,\ell (C))$ and
$\ell(C)<p/2$. Let $I_1,\ldots , I_i,\ldots , I_{2t}$ be the
sequence of intervals defined by $I_i = [(i-1)c,ic)$, where $c=\ell
(C)$ and $t<p/2c$. Let $B\subset \Z/p\Z$   such that for every $i=1,
\ldots , 2t$, we have $I_i\cap B\neq \emptyset$. Then,
  $$|B+C|\geq |B\cup [(B+C)\cap I]|
    \geq |B| +(t-\frac 12)\ell(C)\left(\rho C - \frac{|B\cap I|}{(2t-1)c}\right),$$
where $I= I_1\cup\ldots\cup I_{2t}$.
\end{lemma}

\begin{proof} Let $B'=B\cap I$.
Let $B_0^i = B'\cap I_{2i-1}$ and $B_1^i = B'\cap I_{2i}$ and
define
 $B_0' = \bigcup_{i=1}^tB_0^i$, $B_1' = \bigcup_{i=1}^tB_1^i$
so that $B'=B_0'\cup B_1'$. Note that, since $C\subset [0,c)$,
 $$(B_0^i+C)\cap (B_0^j+C) = \emptyset $$
 for $i\neq j$ and that $B_0^i+C\subset I_{2i-1}\cup I_{2i}$. Therefore $B_0'+C$
 can be written as the following
 union of disjoint sets.
 $$B_0'+C = \bigcup_{i=1}^t (B_0^i+C)\subset I_1\cup\ldots\cup I_{2t}.$$
 Hence, since every set $B_0^i$ is nonempty, the Cauchy-Davenport
 Theorem implies
 \begin{equation}
   \label{eq:B_0'}
   |B_0'+C|\geq |B_0'| + t(|C|-1).
 \end{equation}
In a similar manner we have
\begin{eqnarray*}
  (B_1'+C)\cap I &=& \bigcup_{i=1}^{t-1} (B_1^i+C) \;\;\cup\;\;
  (B_1^{2t}+C)\cap I\\
                 &\supset& \bigcup_{i=1}^{t-1} (B_1^i+C) \;\;\cup B_1^{2t}
\end{eqnarray*}
so that,  applying the Cauchy-Davenport Theorem for $i=1\ldots t-1$,
we get
\begin{equation}
  \label{eq:B_1'}
  |(B_1'+C)\cap I|\geq |B_1'| + (t-1)(|C|-1).
\end{equation}
Now we have $|B+C|\ge |B\setminus B'|+|(B'_0+C)\cap I|$ and likewise
$|B+C|\ge |B\setminus B'|+|(B'_1+C)\cap I|$, hence, applying
\eqref{eq:B_0'} and \eqref{eq:B_1'},
 \begin{eqnarray*}
   |B+C|&\ge &|B\setminus B'|+\frac 12 \left(|(B'_0+C)\cap I|
                                        +|(B'_1+C)\cap I|\right)\\
        &\geq& |B| -|B'|/2+(t-\frac 12)(|C|-1)\\
        &\geq& |B| +(t-\frac 12)c\left(\rho C -\frac{|B'|}{(2t-1)c}\right)
 \end{eqnarray*}
 which proves the result.
 \end{proof}

Lemma \ref{lem:dens} allows us to conclude the proof when the
$(m+3)$--atom $A$ is contained in a short arithmetic progression.

 \begin{lemma}\label{lem:rects1}
Suppose $6\leq m \leq \epsilon |S|$ with $\epsilon\leq 10^{-4}$.
Suppose furthermore that $\ell (A)\le (2-\delta)(|A|-1)$.
 Then $\ell(S)\leq |S|+m+1$.
\end{lemma}

 \begin{proof} Set $a=\ell (A)$. Write $p=2ta+r$, $0<r<2a$ and
let $I_1,\ldots ,I_i,\ldots , I_{2t}$ be the partition of $[0,2ta)$
into the intervals  $I_i = [(i-1)a,ia)$ and $I=\cup_{i=1}^{2t}I_i$.
Let $S'=S\cap I$.

Suppose that $\ell_A (S)+\ell (A)> p$. Then we
have $I_i\cap S'\neq \emptyset$ for each $i=1,\ldots 2t$. By Lemma
\ref{lem:dens} with $B=S$ and $C=A$,
\be\label{eq:s'a}
|S+A|\ge |S|
+(t-\frac 12)a\left(\rho A - \frac{|S'|}{(2t-1)a}\right).
\ee

Now we have $(2t-1)a > p-3a$ by definition of $t$. Since $|A|\le
3m+5$ we have $a=\ell(A) \leq 2(|A|-1)\leq 6m+8$, and since we have
supposed $m\geq 6$, we get $a\leq 8m$. We therefore have
\begin{equation}
  \label{eq:p-3a}
  (2t-1)a> p-3a\geq p-24m> (1-12\epsilon)p.
\end{equation}

By the hypothesis of the Lemma we have $\rho A\geq 1/(2-\delta)$.
Together with \eqref{eq:p-3a} we get, writing $|S'|\leq |S|< p/2$,
$$\rho A - \frac{|S'|}{(2t-1)a} >
    \frac{1}{2-\delta}- \frac{1}{2-24\epsilon}.$$
Finally, applying again \eqref{eq:p-3a}, inequality
\eqref{eq:s'a}
becomes
\begin{equation}
  \label{eq:s'a2}
  |S+A|> |S| + \frac{p}{2}(1-12\epsilon)\left(
  \frac{1}{2-\delta}- \frac{1}{2-24\epsilon}\right).
\end{equation}

Now recall that by definition of
$A$ we have $|A|\geq m+3$. We will therefore get that \eqref{eq:s'a2}
contradicts \eqref{eq:ubsa} whenever the righthand side of
\eqref{eq:s'a2} is greater than $|S|+2|A|$.  Since
$|A|\leq 3m+5\leq 4m\leq 2\epsilon p$, a contradiction is obtained
whenever
\begin{equation}
  \label{eq:4epsilon}
  \frac{1}{2}(1-12\epsilon)\left(
  \frac{1}{2-\delta}- \frac{1}{2-24\epsilon}\right)
  \geq 4\epsilon .
\end{equation}

For $\epsilon\le 10^{-4}$ the inequality \eqref{eq:4epsilon} is
verified for every $\delta>5\cdot 10^{-3}$. Since Lemma
\ref{lem:rectabis} allows us to  choose $\delta$ up to the value
$10^{-1}$,  the hypothesis $\ell_A (S)+\ell (A)> p$ can not hold, so
that the result follows from Lemma~\ref{lem:l(S)+l(A)}.
\end{proof}

\subsection{Structure of $S$ when $\ell (2A)$ is small.}\label{subsec:l(2A)}

To conclude the proof of Theorem \ref{thm:main} it remains to
consider the case where $\ell (A)>(2-\delta)(|A|-1)$.
We break up the proof into several lemmas.

\begin{lemma}\label{lem:ell(A)}
  Suppose $6\leq m \leq \epsilon |S|$ with $\epsilon\leq 10^{-4}$.
  Suppose furthermore that $\ell (A)>(2-\delta)(|A|-1)$. Then
  \begin{itemize}
  \item[(i)] $|2A|\ge (3-\delta)(|A|-1)$.
  \item[(ii)] $\ell (A)\le (1-\delta/2)|2A|.$
  \end{itemize}
\end{lemma}

\begin{proof}
  By point (a) of the final argument in the proof of Lemma~\ref{lem:rectabis} we know that
  we can not have $|2A|< (3-\delta)(|A|-1)$. This proves (i).

  Since $A$ is contained in an arithmetic progression of length
  less than $p/2$ (Lemma~\ref{lem:recta}) we have
  $\ell(A)\leq (\ell(2A)+1)/2$. Now
  Lemma~\ref{lem:rectabis} implies
  $\ell (2A)\leq (2-\delta)(|2A|-1)$,
  hence $(\ell(2A)+1)/2\leq (1-\delta/2)|2A|$. This proves (ii).
\end{proof}

Next we apply the Pl\"unecke-Ruzsa inequalities to exhibit a subset
$T$ of $S$ that sums to a small sumset with $2A$. We then show
that this set $T$ must be contained in an arithmetic progression
with few missing elements.

\begin{lemma}\label{lem:T}
  Suppose $6\leq m \leq \epsilon |S|$ with $\epsilon\leq 10^{-4}$.
  Suppose furthermore that $\ell (A)>(2-\delta)(|A|-1)$. Then there exists
  $T\subset S$ such that, denoting $\lambda=|T|/|S|$,
  \begin{align}
    |2A| & \leq \lambda(4+10\epsilon)(|A|-1), \label{eq:2a}\\
    \ell (T)&\le |T|+2\ell(A). \label{eq:ls}
  \end{align}
\end{lemma}

\begin{proof}
  By Theorem \ref{thm:pr} and (\ref{eq:ubsa}),  there is $T\subset S$
such that
$$
|T+2A|\le (1+\frac{|A|+m}{|S|})^2|T|\le |T|+
2(|A|+m)\frac{|T|}{|S|}+\frac{(|A|+m)^2}{|S|}\frac{|T|}{|S|}.
$$
Writing $|A|+m\leq 3m+5+m\leq 5m\leq 5\epsilon |S|$ and
$\lambda=|T|/|S|$ we get
\begin{equation}   \label{eq:ubs2abis}
|T+2A|\leq |T|+\lambda(|A|+m)(2+5\epsilon)<p.
\end{equation}

Now apply  the Cauchy-Davenport Theorem $|T+2A|\geq |T|+|2A|-1$ in
\eqref{eq:ubs2abis}  to get, since $|A|\geq m+3$,
\begin{eqnarray}
  |2A| -1 &\leq& \lambda(2|A|-3)(2+5\epsilon),\mbox{ and } \nonumber\\
    |2A|  &\leq& 2\lambda(2+5\epsilon)(|A|-1) - \lambda(2+5\epsilon)
    +1. \label{eq:+1}
\end{eqnarray}
Notice that if $\lambda(2+5\epsilon)<1$ then \eqref{eq:+1} gives
$|2A|< 2(|A|-1)+1$ which contradicts the Cauchy-Davenport Theorem.
Therefore we have $1-\lambda(2+5\epsilon)\leq 0$ and \eqref{eq:+1}
yields \eqref{eq:2a}.

In the remaining part we prove \eqref{eq:ls}.  Recall that the
hypothesis of the present lemma together with
Lemma~\ref{lem:rectabis} imply
\begin{equation}
  \label{eq:l(2A)}
  \ell(2A)\leq (2-\delta)(|2A|-1).
\end{equation}
Suppose first that
\be\label{eq:nohole} \ell_{2A} (T)+\ell (2A) > p. \ee
 Set
$a_2=\ell (2A)$ and $p=2ta_2+r$ with $0<r<2a_2$. Let
$I=I_1\cup\cdots\cup I_{2t}$ with $I_i=[(i-1)a_2,ia_2)$. By
(\ref{eq:nohole}) we have $T\cap I_i\neq\emptyset$ for each
$i=1,\ldots ,2t$. By Lemma \ref{lem:dens} with $B=T$ and $C=2A$,
\begin{equation}
  \label{eq:s'+2a}
|T+2A|\ge |T|+(t-\frac 12)a_2\left(\rho(2A)-\frac{|T'|}{(2t-1)a_2}\right)
\end{equation}
where $T'=T\cap I$. By   \eqref{eq:l(2A)} we have $a_2\le 2|2A|$, so
that by using   \eqref{eq:2a} and $\lambda \leq 1$ we obtain the
following rough upper bound
$$a_2\leq (8+20\epsilon)|A|\leq 9(3m+5)\leq 36m$$
where we have used $\epsilon \leq 1/20$.

As in the proof of Lemma \ref{lem:rects1}, we have, by definition of
$t$,
\begin{equation}
  \label{eq:p-3a_2}
  (2t-1)a_2\geq p -3a_2 \geq p-108m \geq p(1-54\epsilon)
\end{equation}
so that, writing $|T'|\leq |T|\leq |S|\leq p/2$, and applying
\eqref{eq:l(2A)} we have
\begin{eqnarray*}
\rho(2A)-\frac{|T'|}{(2t-1)a_2}
  \geq
  \frac{1}{2-\delta}-\frac{1}{2-108\epsilon}.
\end{eqnarray*}
Applying again \eqref{eq:p-3a_2}, inequality \eqref{eq:s'+2a} becomes
\begin{equation}
  \label{eq:s'+2abis}
  |T+2A|\ge |T|+\frac p2(1-54\epsilon)
  \left(\frac{1}{2-\delta}-\frac{1}{2-108\epsilon}\right).
\end{equation}
On the other hand,  \eqref{eq:ubs2abis} implies
$$|T+2A|\leq |T| + 10m +25\epsilon m\leq |T|+p(5\epsilon +25\epsilon^2/2)$$
which together with \eqref{eq:s'+2abis} gives
  \begin{equation}
  \label{eq:contradict}
5\epsilon +25\epsilon^2/2\geq \frac 12(1-54\epsilon)
     \left(\frac{1}{2-\delta}-\frac{1}{2-108\epsilon}\right).
\end{equation}
  For $\epsilon =10^{-4}$  the inequality \eqref{eq:contradict} fails to hold for each $\delta\ge 2\cdot 10^{-2}$. Since
  \eqref{eq:l(2A)} holds for every $\delta \le 10^{-1}$,   the hypothesis \eqref{eq:nohole}
  can not hold, so that the sumset $T+2A$ behaves like a sum of integers.
  Let us write
  $$|T+2A| = |T| + |2A| + \mu$$
  and check that the conditions of Theorem~\ref{thm:ls2} hold.
  By Lemma~\ref{lem:ell(A)} (i) we have
  \begin{align*}
    |2A|&\ge (3-\delta)(|A|-1)\\
        &\geq (2+5\epsilon)|A| +(1-\delta -5\epsilon)|A| -3\\
        &\geq (2+5\epsilon)|A| + \frac 32
   \end{align*}
since $m\geq 6$ and $|A|\geq m+3\geq 9$. Therefore
  \begin{align*}
   2|2A|&\geq 2(2+5\epsilon)|A| + 3\\
        &\geq (2+5\epsilon)(|A|+m) + 3,
  \end{align*}
which, since $\mu \leq (|A|+m)(2+5\epsilon) -|2A|$ by
\eqref{eq:ubs2abis}, leads to
\begin{equation}
  \label{eq:mu+3}
  |2A| \geq \mu +3.
\end{equation}
Now by definition of $\lambda$ we have $|T|=\lambda |S|$ and we also have
$|S|\geq 11\epsilon |S|$, so that
\begin{align*}
  |T| & \geq \lambda 11\epsilon |S| \geq \lambda 11m\\
      & \geq \lambda (2+5\epsilon)5m \geq  \lambda (2+5\epsilon)(|A|+m)
\end{align*}
and, since   $\mu \leq \lambda(|A|+m)(2+5\epsilon) -|2A|$ by
\eqref{eq:ubs2abis}, we obtain
\begin{equation}
  \label{eq:mu+4}
  |T| \geq \mu + |2A| \geq \mu +4.
\end{equation}
Inequalities \eqref{eq:mu+3} and \eqref{eq:mu+4} mean that
Theorem~\ref{thm:ls2} holds and we have~:
  $$\ell (T)\le |T|+\mu+1\le |T|+|2A|\le |T|+\ell(2A)\le |T|+2\ell(A).$$
This proves \eqref{eq:ls} and concludes the lemma.
\end{proof}

\begin{lemma}\label{lem:last}
  Suppose $6\leq m \leq \epsilon |S|$ with $\epsilon\leq 10^{-4}$.
  Suppose furthermore
  that $\ell (A)>(2-\delta)(|A|-1)$. Then $\ell(S)\leq |S|+m+1$.
\end{lemma}

\begin{proof}
Let $T$ be the set guaranteed by Lemma~\ref{lem:T}.
Let $\ot=S\setminus T$, which belongs to an
interval of length $p-\ell (T)$. Set $a=\ell(A)$.
Let us apply again Lemma~\ref{lem:dens}, this time with
$B=S$, $C=A$, and $t$ defined so as to have
$p-\ell(T)= 2ta+r$,
$0\le r<2a$. As before, set
$I=I_1\cup\cdots\cup I_{2t}$ with $I_i=[(i-1)a,ia)$.
Note that $T\cap I=\emptyset$, so that $\ot\cap I = S\cap I$.
Let us first suppose
\begin{equation}
  \label{eq:l(S)+l(A)}
  \ell_A(S)+\ell(A)> p
\end{equation}
which implies
$\ot\cap I_i\neq\emptyset$ for every $i=1,\ldots ,2t$, so that by
Lemma~\ref{lem:dens}, and denoting $\ot'=\ot\cap I= S\cap I$,
\begin{eqnarray}
 |S+A| & \geq & |S\cup [(S+A)\cap I]|\nonumber\\
       & \geq & |S| + (t-\frac 12)a\left(\rho
         A-\frac{|\ot'|}{(2t-1)a}\right). \label{eq:dens}
\end{eqnarray}

By definition of $t$ and by \eqref{eq:ls} we have
\begin{equation}
  \label{eq:-5a}
  (2t-1)a > p-\ell(T) - 3a \geq p -|T|-5a.
\end{equation}
Now Lemma~\ref{lem:ell(A)} (ii) and \eqref{eq:2a} give the following
upper bound on $a$
  $$a\leq |2A|\leq \lambda(4+10\epsilon)|A|\leq \lambda(4+10\epsilon)4m
    \leq \lambda(4+10\epsilon)2\epsilon p$$
so that we can write $-5a\geq -\lambda f(\epsilon)p$ with
$f(\epsilon) = 10(4+10\epsilon)\epsilon$.
Writing $|T|=\lambda |S| < \lambda p/2$, \eqref{eq:-5a} becomes
\begin{equation}
  \label{eq:(2t-1)a}
  (2t-1)a > p(1-\lambda(\frac 12 +f(\epsilon)))
\end{equation}

Next we write $|\ot '|\leq |\ot|=|S|-|T| = (1-\lambda)|S|$, so that
$|S|\leq p/2$ gives
\begin{equation}
  \label{eq:ot'}
  |\ot '|\leq \frac p2(1-\lambda)
\end{equation}

Finally we bound $\rho A$ from below.
Apply again Lemma~\ref{lem:ell(A)} (ii) and \eqref{eq:2a} to get
$$\ell(A) \leq (1-\delta/2)|2A|\leq
(1-\delta/2)\lambda(4+10\epsilon)(|A|-1),$$
so that we have
\begin{equation}
  \label{eq:rhoA}
  \rho A \geq \frac{1}{\lambda(1-\delta/2)(4+10\epsilon)}.
\end{equation}
Applying \eqref{eq:(2t-1)a}, \eqref{eq:ot'} and \eqref{eq:rhoA} to
\eqref{eq:dens} now gives
$$|S+A|>|S|+\frac p2
\left[
\frac{1-\lambda(\frac 12+f(\epsilon))}{\lambda(1-\delta/2)(4+10\epsilon)}
-\frac 12 (1-\lambda)
\right].
$$
Together with \eqref{eq:ubsa}, writing $|A|\leq 4m$ and $m\leq\epsilon
p/2$, we obtain
\begin{equation}
  \label{eq:contradict2}
  \frac{1-\lambda(\frac 12+f(\epsilon))}{\lambda(1-\delta/2)(4+10\epsilon)}
-\frac 12 (1-\lambda) -5\epsilon <0.
\end{equation}
Now there exists $\epsilon_\delta>5.8\; 10^{-3}>0$ such that for every
$\epsilon\leq\epsilon_\delta$, the lefthandside of
\eqref{eq:contradict2} is strictly positive for every value of
$\lambda\in [0,1]$.
In that case \eqref{eq:contradict2} can not hold and we obtain a
contradiction with the hypothesis \eqref{eq:l(S)+l(A)}. Therefore
Theorem \ref{thm:ls2} implies the result.
\end{proof}

{\bf Numerical values:} As it has been shown in the proofs Theorem
\ref{thm:main} holds with $\epsilon = 10^{-4}$. As for the value of $p_0$,
we use $m\ge 6$ in Section 4, so in order to cover smaller
values of $m$, the prime $p$ should satisfy the condition in Lemma
\ref{lem:rectsasmall} that $\log_4 p\ge 6m+11\ge 47$ which is
equivalent to $p\geq 2^{94}$.
We have tried to strike a balance between readability and obtaining
the best possible constants. These values of $\epsilon$ and $p_0$ 
are not the best possible, but they give a reasonable account of what
can be achieved through the methods of this paper.

\end{document}